\documentclass{article}

\usepackage{amsthm}
\usepackage{epsfig}
\newtheorem{proposition}{Proposition}
\newtheorem{corollary}{Corollary}
\newtheorem{theorem}{Theorem}

\title{On perfect matchings, edge-colourings and eigenvalues of cubic graphs}

\author{Willem H. Haemers\thanks{e-mail haemers@uvt.nl}
\\
{\it\small Tilburg University, The Netherlands}
}
\date{}
\begin{document}
\maketitle
\begin{abstract}
We discuss the question whether the existence of perfect matchings in a cubic graph can be seen from the spectrum of its adjacency matrix. 
For regular graphs in general and
for three edge-disjoint perfect matchings in a cubic graph (that is, an edge-colouring with three colours) 
the answer is known to be negative.
In the latter case, a few counter examples (found by computer) are known.
Here we show that these counter examples can be extended to an infinite family by use of truncation.
Thus we obtain infinitely many pairs of cospectral cubic graphs with different edge-chromatic number.
For all these pairs both graphs have a perfect matching, and the mentioned question is still open.
%But we do find a new sufficient condition for a perfect matching in a cubic graphs in terms of its spectrum.
%In addition
We obtain a few more results concerning spectral characterisations of cubic graphs, and point out a serious mistake in Section~3 of an earlier version of this paper.
\end{abstract}

\section{Introduction}

%The spectrum of a graph is the multiset of the eigenvalues of the adjacency matrix. 
We consider the relation between the structure of a graph and the spectrum of the adjacency matrix. 
More in particular, we are interested in graph properties which are characterised by its spectrum.
Some well-known examples are: the numbers of vertices, edges and triangles, bipartiteness and regularity 
(for this and other background on graph spectra we refer to~\cite{BH} and~\cite{DH}).
Here we focus on regular graphs.
For several graph properties, pairs of cospectral regular graphs have been found, where one graph has the property and the other one not, proving that the property is not characterised by the spectrum, not even if we require regularity. 
We mention: the diameter~\cite{HS}, the chromatic number~\cite{EH}, edge and vertex-connectivity~\cite{H}, being Hamiltonian~\cite{EH} and having a perfect matching~\cite{BCH}.
We remark that none of the mentioned counter examples are cubic (regular of degree~3).
In this note we find pairs of cospectral cubic graph that can serve as counter examples to the spectral characterisation of some famous graph problems.

An {\em edge-colouring} of a graph $G$ is a colouring of its edges such that intersecting edges
have different colours. 
Thus a set of edges with the same colours (called a colour class) is a matching.
The {\em chromatic index} (also known as the {\em edge-chromatic number}) of $G$ is the minimum number of colours in an edge-colouring. 
For a regular graph $G$ of degree $k$, it follows from Vizing's theorem \cite{V} that the chromatic index of $G$ is equal to $k$ or $k+1$.
Moreover, if the chromatic index equals $k$, then each colour class is a perfect matching.

In \cite{EH} the question is raised whether the chromatic index 
of a regular graph is determined by the spectrum of the adjacency matrix.
Shortly after that Yan and Wang~\cite{YW} found a pair of cospectral cubic graphs on 16 vertices with different chromatic index by computer search; see Figure~\ref{GH}.
They also established that it is the unique such pair of order at most 16.
Also by computer, a few more pairs on 18 and 20 vertices were found by Krystal Guo (private communication).
Here we show that a simple operation called truncation leads to infinitely many examples.
\begin{figure}[ht]
\begin{center}
\epsfig{file = 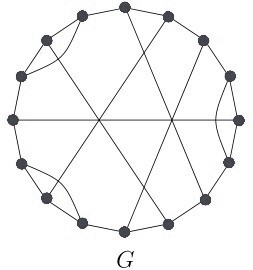,height=135pt, width=125pt}\hspace{30pt}
\epsfig{file = 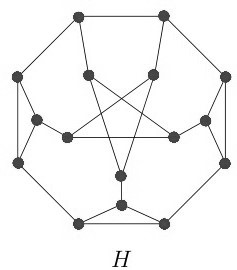,height=135pt, width=125pt}
\caption{Cospectral cubic graphs with chromatic index 3 and 4, 
respectively, and characteristic polynomial 
$x(x-3)(x+2)(x^2-2)(x^2-x-3)(x^3-4x-2)(x^3-4x+1)(x^3+2x^2-2x-2)$ 
}\label{GH}
\end{center}
\end{figure}
~\\[-30pt]

\section{Truncation}

Suppose $G$ is a cubic graph. 
{\em Truncation} is the operation that replaces a vertex $v$ of $G$ by a triangle and connects each neighbour of $v$ to one vertex of the triangle (see Fig~\ref{truncation}). 
\begin{figure}[ht]
%\begin{center}
\setlength{\unitlength}{3pt}
\begin{picture}(0,20)(-52,-10)
\thicklines
\put(0,0){\line(-4,3){8}}
\put(0,0){\line(4,3){8}}
\put(0,-8){\line(0,1){8}}
\put(27,-8){\line(0,1){4}}
\put(27,-4){\line(3,5){5}}
\put(27,-4){\line(-3,5){5}}
\put(22,4){\line(1,0){9}}
\put(22,4){\line(-2,1){4}}
\put(32,4){\line(2,1){4}}
\put(10,-2){\huge $\rightarrow$}
\put(0,0){\circle*{2}}
\put(27,-4){\circle*{2}}
\put(22,4){\circle*{2}}
\put(32,4){\circle*{2}}
\end{picture}\caption{Truncation}\label{truncation}
%\end{center}
\end{figure}
If we truncate with respect to every vertex we obtain the truncation of $G$ denoted by $T(G)$.
The truncated graph $T(G)$ can also be described as the line graph of the complete subdivision of $G$.
Using this the spectrum of $T(G)$ can be expressed in terms of the spectrum of $G$;
see Theorem 2.1 in Zhang, Chen and Chen~\cite{ZCC}.

\begin{theorem}\label{trunc-eig}\cite{ZCC}
Let $G$ be a cubic graph of order $n$ with spectrum $\lambda_1=3\geq\ldots\geq\lambda_n$.
Then the eigenvalues of the truncated graph $T(G)$ are 
$\frac{1}{2}\pm\frac{1}{2}\sqrt{13+4\lambda_i}$ for 
$i\in\{1,\ldots,n\}$,
%$i=1,\ldots,n$, 
$-2$ and $0$, both with multiplicity $n/2$.
\end{theorem}

It is easily seen that two truncated cubic graphs $T(G)$ and $T(G')$ are isomorphic if and only if $G$ and $G'$ are. 
So if we start with a pair of nonisomorphic cospectral cubic graphs, then by repeated truncation we find infinitely many such pairs. 
Ramezani and Tayfeh-Rezaie~\cite{FT} found several such pairs with 14, 16, 18 and 20 vertices.

\begin{proposition}~\label{trunc-index}
A cubic graph $G$ has the same chromatic index as its truncated graph $T(G)$.
\end{proposition}

\begin{proof}
Suppose $G$ is a cubic whose edges are properly coloured with three colours.
The edges of $T(G)$ between triangles correspond to the edges of $G$ and we will give
them the same colour as in $G$. 
Then for each triangle $\Delta$ of $T(G)$, the three edges 
that meet $\Delta$ in one vertex have a different colour.
So we obtain an edge-colouring of $T(G)$ with three colours when we colour each edge of $\Delta$ with the colour of the edge meeting $\Delta$ in the opposite vertex.

Conversely, if $T(G)$ has an edge-colouring with three colours, then for each triangle $\Delta$ the three edges meeting $\Delta$ in one vertex have different colours and give an edge-colouring of $G$ with three colours.
\end{proof}

From the above proof it is clear that the chromatic index also doesn't change if we truncate with respect to just some of the vertices. 
If we do so for the vertices of a path of length 2 in the Petersen graph,
we obtain graph $H$ in Figure~\ref{GH}.
Therefore $H$ has the same chromatic index as the Petersen graph, which is equal to 4.
  
Theorem~\ref{trunc-eig} and Proposition~\ref{trunc-index} imply that if we start repeated truncation with one of the pairs of cospectral cubic graphs 
with different chromatic index, we can conclude: 

\begin{theorem}\label{cospcub}
There exist infinitely many pairs of cospectral cubic graphs with different chromatic index.
\end{theorem}

The chromatic index of a graph is the same as the chromatic number of its line graph.
Moreover, the line graphs of non-isomorphic cospectral regular graphs are also non-isomorphic and cospectral. 
Because the line graph of a cubic graph is regular of degree 4, we have:

\begin{corollary}\label{cosp4reg}
There exist infinitely many pairs of cospectral 4-regular graphs with different chromatic number.
\end{corollary}

It is also not difficult to see that a cubic graph $G$ is Hamiltonian if and only if $T(G)$ is.
In Figure~\ref{GH}, $G$ is Hamiltonian and $H$ is not (another such pair of cubic graphs can be found in~\cite{LWYL}), therefore by repeated truncation we get the following result.

\begin{proposition}
There exist infinitely many pairs of cospectral cubic graphs, where one is Hamiltonian and the other one not.
\end{proposition}

Another interesting observation about truncation is the following:

\begin{proposition}
Being a truncated cubic graph is a property characterised by the spectrum.
\end{proposition}

\begin{proof}
Suppose $H$ is a graph with the spectrum of $T(G)$ given above.
Then $H$ is a cubic graph of order $3n$ with $n$ triangles.
Moreover, $H$ has the same number of closed 4-walks as $T(G)$.
But $T(G)$ has no 4-cycles and therefore only trivial closed 4-walks.
So also $H$ has no 4-cycles, which implies that all triangles are 
disjoint, and that two triangles are connected by at most one edge.
Now we define the graph $G'$ whose vertices are the triangles, 
where triangles are adjacent whenever they are connected by an edge in 
$H$. 
Then clearly $H=T(G')$.
\end{proof}

It follows that if a cubic graph $G$ is determined by the spectrum, then so is $T(G)$.
According to~\cite{FT}, there are more than half a million non-isomorphic cubic graphs of order 20 which are determined by their spectrum. 
So by repeated truncation we obtain equally many infinite sequences of cubic graphs determined by the spectrum. 

\section{Perfect matching}

In the previous section we saw that having three edge-disjoint perfect matchings in a cubic graph is a property that can not be seen from the spectrum. 
But the question remains if the existence of a perfect matching in a cubic graph can be seen from the spectrum. 
Note that truncation cannot help, because every truncated cubic graph has a perfect matching.
We believe the answer should also be negative, but failed to prove it.
%However, along the way we did find a new sufficient condition for existence of a perfect matching in a cubic graph in terms of the spectrum.
We know the following sufficient conditions, but they are far from necessary. 
\begin{figure}[ht]
%\begin{center}
\setlength{\unitlength}{3pt}
\begin{picture}(30,14)(-60,0)
\thicklines
\put(0,0){\line(0,1){12}}
\put(0,0){\line(1,0){12}}
\put(0,0){\line(1,1){12}}
\put(0,12){\line(1,-1){12}}
\put(0,12){\line(1,0){12}}
\put(18,6){\line(-1,1){6}}
\put(18,6){\line(-1,-1){6}}
\put(0,12){\circle*{2}}
\put(12,12){\circle*{2}}
\put(12,0){\circle*{2}}
\put(0,0){\circle*{2}}
\put(18,6){\circle*{2}}
%\put(7,-4){$F$}
\end{picture}\caption{Graph with largest eigenvalue $\theta\approx 2.85577$}\label{F}
\end{figure}
\begin{theorem}\label{PM}
Let $\theta$ ($\approx 2.85577$) be the largest eigenvalue of the graph in Figure~\ref{F}, and let $G$ be a cubic graph of order $n$ with eigenvalues $\lambda_1 =3\geq \lambda_2 \geq\ldots\geq\lambda_n$.
If $\lambda_3<\theta$, or $\lambda_n=-3$ and $\lambda_2 <3$, then $G$ has a perfect matching.
\end{theorem}
For the first condition we refer to~\cite{CGH} and the second condition follows from K\"onig's theorem applied to a cubic bipartite graph.
\\[5pt]
{\bf Erratum.}
The published (and previous arXiv) version of the present paper~\cite{H'}, gives another sufficient spectral condition for existence of a perfect in a cubic graph. 
However, the proof is incorrect and the result should be ignored.
\section{Final remarks}
In the investigation of properties characterised by the spectrum of $k$-regular graphs, 
the cubic graphs play a special role.
If $k\leq 2$, every k-regular graph is determined by its spectrum.
For $k\geq 3$ there exist many pairs of $k$-regular cospectral graphs that disprove spectral characterisation of various famous graph properties.
However, if $k=3$ we only know two such properties (being Hamiltonian and having chromatic index $k$).

Our experience is that finding cubic counter examples to spectral characterisations is difficult, and sometimes even impossible.
For example, there exist infinitely many pairs of cospectral connected $k$-regular graphs with different chromatic number when $k = 4$ (Corollary~\ref{cosp4reg}) and for infinitely many larger values of $k$ (see~\cite{EH}), but when $k=3$ such a pair cannot exist.

\begin{proposition}
The chromatic number of a connected cubic graph is determined by its spectrum.
\end{proposition}

\begin{proof}
Let $G$ be a connected cubic graph with chromatic number $\chi(G)$.
According to Brooks' theorem~\cite{B} $\chi(G)\leq 4$ with equality if and only if $G=K_4$, and $\chi(G)=2$ whenever $G$ is bipartite.
Bipartiteness and being $K_4$ can be seen from the spectrum, therefore the spectrum of $G$ determines $G$ when $\chi(G)= 2$ and when $\chi(G)=4$, and therefore also when $\chi(G)=3$.
\end{proof}
Similarly, the clique number (order of the maximum clique) of a connected cubic graph is determined by its spectrum, whilst it is not true for a regular graph in general (see \cite{EH}).
%Similarly, the clique number (order of the maximum clique) of a regular graph cannot be seen from the spectrum in general (see \cite{EH}), but for a connected cubic graph it can.
In both cases the connectivity requirement is essential.
Indeed, the disjoint union of a cube and two truncated $K_4$'s is cospectral with
the disjoint union of two $K_4$'s and the bipartite double of a truncated $K_4$.
Both graphs have spectrum $\{3^3,2^6,1^3,0^4,-1^9,-2^6,-3\}$.
The first one has chromatic and clique number equal to 3, 
and for the second graph both numbers are equal to 4.
\\

\noindent{\bf Acknowledgement.}
I thank Krystal Guo for her initial help with the research for this note.

\end{document}